\documentclass[reqno]{amsart}
\usepackage{nicefrac}  
\usepackage{hyperref}
\usepackage{enumerate}

\newcommand{\on}[1]{\operatorname{#1}}

\newcommand{\R}{\mathbb{R}}

\newcommand{\gt}{\tilde{g}}
\newcommand{\F}{\tilde{\mathcal{F}}}

\newtheorem{proposition}{Proposition}
\newtheorem{corollary}{Corollary}
\newtheorem{lemma}{Lemma}
\newtheorem{remark}{Remark}
\newtheorem{fact}{Fact}

\title[Isoperimetric Inequalities For Minimal Submanifolds]{ISOPERIMETRIC INEQUALITIES FOR MINIMAL SUBMANIFOLDS IN RIEMANNIAN MANIFOLDS: \\A COUNTEREXAMPLE IN HIGHER CODIMENSION}

\author{ Victor Bangert and Nena R\"ottgen}        
\address{Mathematisches Institut, Abteilung f\"ur Reine Mathematik, Albert-Ludwigs-Universit\"at, Eckerstr. 1, 79104 Freiburg im Breisgau, Germany}               

\date{December 23, 2010}                                     
\subjclass[2000]{49Q20, 53C22, 49Q05, 53C42} 
\keywords{Riemannian manifolds, geodesics, isoperimetric inequalities, stationary varifolds}

\begin{document}

\maketitle

\begin{abstract} \hfill \newline
For compact Riemannian manifolds with convex boundary, B.~White proved the following alternative: Either there is an isoperimetric inequality for minimal hypersurfaces or there exists a closed minimal hypersurface, possibly with a small singular set. There is the natural question if a similar result is true for submanifolds of higher codimension. Specifically, B.~White asked if the non--existence of an isoperimetric inequality for $k$-varifolds implies the existence of a nonzero, stationary, integral $k$-varifold. We present examples showing that this is not true in codimension greater than two. The key step is the construction of a Riemannian metric on the closed four--dimensional ball $B^4$ with the following properties: (1) $B^4$ has strictly convex boundary. (2) There exists a complete nonconstant geodesic $c:\mathbb{R} \to B^4$. (3) There does not exist a closed geodesic in $B^4$.
\end{abstract}

\section{Introduction}

If $D$ is a two-dimensional Riemannian disc with locally convex boundary $\partial D$ and if there is no closed geodesic in $D$, then there is a constant $C>0$ such that every geodesic segment in $D$ has length at most $C$. An equivalent formulation of this fact is: If there exists a nonconstant geodesic $c:\mathbb{R} \to D$, then $D$ contains a closed geodesic. This fact is due to Birkhoff, cf. \cite[ VI. 10]{Birkhoff}, and played a role in the proof that there exist infinitely many closed geodesics on every Riemannian $2$-sphere, cf. \cite{Bangert} and \cite{Franks}. 

In arbitrary dimensions, an analogous result has been proven by B.\ White \cite[Theorem 2.1]{White} in the codimension one situation, i.e. when geodesics are replaced by minimal hypersurfaces. As part of the opening colloquium of the collaborative research center SFB/Transregio 71 in Freiburg, April 2009, B.\ White lectured on this result, and posed the question if there could be a version of the result that is not restricted to the codimension one case, see also \cite[Remark 2.8]{White}. 

Here, we construct a Riemannian metric $g$ on the closed four-dimensional ball $B^4$ such that $\partial B^4$ is strictly convex and such that $B^4$ carries a complete geodesic, but no closed geodesic. Actually one would expect that such an example exists already on the closed $3$--ball. We believe that this is the case, but our construction would be considerably more complicated.
\newpage
Now we explain how this can be used to answer B.~White's question \cite[Remark 2.8]{White}, that explicitly asks:

\vspace{1ex}
\noindent\emph{Let $N$ be a compact, $k$-convex Riemannian manifold containing a nonzero, stationary $k$-varifold. Does this imply that $N$ contains an \textbf{integral} stationary $k$-varifold?}

\vspace{1ex}
\noindent For more details on this question see Section \ref{sec:appl}. 

Taking the Riemannian product of an arbitrary closed Riemannian manifold $M$ with our example $(B^4,g)$ we obtain a compact manifold $\tilde{M}$ of dimension $\tilde{m}\geq 4$. This $\tilde{M}$ has convex boundary. So $\tilde{M}$ is $k$-konvex for every $k < \tilde{m}$. The product of $M$ with a complete geodesic in $B^4$ gives an $(\tilde{m}-3)$-dimensional, immersed, totally geodesic submanifold isometric to $M \times \mathbb{R}$. This implies that there is no isoperimetric inequality for $(\tilde{m}-3)$-dimensional minimal submanifolds in $\tilde{M}$. Hence, from \cite[Theorem 2.3]{White} we know that there exists a nonzero, stationary $(\tilde{m}-3)$-varifold in $\tilde{M}$. Indeed, we can describe explicitly such a varifold $V_0$ in $M$, and prove that -- up to scale -- $V_0$ is the only stationary $(\tilde{m}-3)$-varifold in $\tilde{M}$.

From the explicit description of $V_0$ we conclude that $V_0$ is not rectifiable and, hence, not integral. This gives a negative answer to B.~White's question for the case of varifolds of arbitrary dimension, and codimension at least three.

Finally we sketch the idea underlying the construction of the metric $g$ on $B^4$. First, we deform the standard metric $g_0$ on the ball $B^4 \subset \mathbb{R}^4$ of radius $2$ so that all the spheres $S^3(\rho) \subset B^4$ of radius $\rho \in ]0,2[$ remain strictly convex, except for $S^3(1)$ whose second fundamental form vanishes precisely on the vectors tangent to an irrational geodesic foliation $\mathcal{F}$ of the Clifford torus $\mathbb{T}^2 \subset S^3(1)$. This implies that there are no closed geodesics in $B^4$ with respect to this metric. Moreover, we achieve that also the second fundamental form of the Clifford torus $\mathbb{T}^2$ vanishes in the direction of $\mathcal{F}$. Then the leaves of $\mathcal{F}$ are complete geodesics not only in $\mathbb{T}^2$ but also with respect to the metric on $B^4$. 

\tableofcontents
\section{Convex Distance Functions} \label{sec:convex}

In this section we will recall some well known facts about geodesics and distance functions. Let $(M,g)$ denote a Riemannian manifold and $i:N \hookrightarrow M$ a submanifold. We will denote the induced metric on $N$ by $g^N$. Then a curve $c:I \subset \mathbb{R} \to N$ is a $g$-geodesic if and only if $c$ is a $g^N$-geodesic and the second fundamental form of $N$ vanishes on its tangent vectors. 

Now, let $F:N\times (-\varepsilon,\varepsilon) \to M$ be a normal variation with variational vector field $V= \frac{\partial F}{\partial t}_{\vert t=0}$ along $i=F(\,\cdot\,,0)$. Then, for any tangent vectors $v_1$, $v_2 \in T_pN$, one calculates, cf. \cite[(1.33)]{Colding},
\begin{equation}
  \label{eq:3}
\frac{d}{d t}_{\vert t=0} (F_t^*g)(v_{1},v_{2})=  g(\nabla_{v_{1}} V, v_{2})+ g(v_{1},\nabla_{v_{2}} V),  
\end{equation}
where $F_t:N \to M$ is defined by $F_t(\,\cdot\,)= F(\,\cdot\,,t)$.

If, additionally, $|V|=1$, it follows from equation \eqref{eq:3}, that the second fundamental form $h^N(\cdot,\cdot)$ of $N$ with respect to $V$ is given by 
 \begin{equation}
  \label{eq:2}
 h^N(v_1,v_2)= - \frac{1}{2} \,\frac{d}{d t}_{\vert t=0} (F_t^*g)(v_{1},v_{2}). 
\end{equation}

We will use this fact in the special case where $N$ is a level set of a $C^\infty$-function $d$ with $|\operatorname{grad} d\,|=1$. These functions will be called distance functions, cf. \cite[2.3.1]{Petersen}. Then the restriction of the gradient flow $\Phi_t$ to $N$ is a normal variation with variational vector field $V=\operatorname{grad} d$. The gradient of $d$ is contained in the null space of the Hessian $\nabla^2 d$  and for any $v_1$, $v_2 \in T_pN$ one obtains

\begin{equation} \label{eq:1}
 \nabla^2 d (v_1,v_2)= g(\nabla_{v_{1}} V, v_{2})= - h^N(v_1,v_2). 
\end{equation}

Hence a distance function is a convex function if the second fundamental form (with respect to $\operatorname{grad} d$) of any of its level sets is everywhere negative semidefinite. Recall that a $C^2$-function $f:M\to \mathbb{R}$ is convex if one of the following equivalent conditions is satisfied:
\begin{itemize} 
\item[\labelitemiv] For any geodesic segment $c:I \to M$ the composition $f \circ c:I \to \mathbb{R}$ is convex.
\item[\labelitemiv] The Hessian $\nabla^2 f$ is everywhere positive semidefinite.
\end{itemize}
In particular, we have:

\begin{fact}\label{fact:konvex} Let $f:M \to \mathbb{R}$ be a convex function. Then any closed geodesic in $M$ is contained in one of the level sets of $f$. If the second fundamental form of a smooth level set of $f$ is definite at some point, then there is no closed geodesic passing through this point. \hfill $\qed$
\end{fact}

Therefore there are no closed geodesics on a manifold that is equipped with a convex distance function, if its Hessian restricted to the tangent spaces of the level sets is everywhere definite. 

\section{The Example}\label{example}

Consider the closed standard 4-ball $(B^4,g_0)$ with radius $2$ and the Clifford torus $(\mathbb{T}^2, g_0^{\mathbb{T}^2})$ given by $\{\nicefrac{1}{\sqrt{2}}(\sin \varphi, \cos \varphi, \sin \theta, \cos \theta) \mid \varphi, \theta \in [0,2 \pi]\}$. The Clifford torus is a flat torus that is isometrically embedded in the standard sphere $S^3 \subset B^4$. The map $(\varphi, \theta)\in \R^2\to \frac{1}{\sqrt{2}}(\sin \varphi, \cos \varphi, \sin \theta, \cos \theta)$ from euclidean $\R^2$ to the Clifford torus $\mathbb{T}^2$ is a homothetic covering map. The projection to $\mathbb{T}^2$ of a family of parallel lines in $\mathbb{R}^2$ will be called a geodesic foliation of $\mathbb{T}^2$. A geodesic foliation of $\mathbb{T}^2$ is called rational  if the corresponding family of parallels has rational slope and irrational otherwise. The geodesics of a rational foliation of $\mathbb{T}^2$ are all closed, while the geodesics of an irrational geodesic foliation of $\mathbb{T}^2$ are all dense on $\mathbb{T}^2$. 

The metric $g$ that we will define on $B^4$ will have the following properties:

\begin{enumerate}[(G1)]
\item \label{metrik1}The induced metric $g^{\mathbb{T}^2}$ on $\mathbb{T}^2$ is the flat one induced by $g_0$.
\item \label{metrik2}The function $d:B^4 \to [0,2]$ given by the euclidean distance to zero is a convex distance function with respect to the metric $g$.
\item \label{metrik3}There exists an irrational geodesic foliation $\mathcal{F}$ of $\mathbb{T}^2$ such that the following holds for the hessian $\nabla^2(d^2)$ with respect to $g$: $\nabla^2(d^2)\vert_x$ is positive definite for all $x \in B^4 \setminus \mathbb{T}^2$, and for $x \in \mathbb{T}^2$ the nullspace of $\nabla^2(d^2)\vert_x$ coincides with the tangent line to $\mathcal{F}$ at $x$.
\item \label{metrik4}The second fundamental form of the Clifford torus $\mathbb{T}^2$ as a submanifold of $S^3$ vanishes on the vectors tangent to the irrational geodesic foliation $\mathcal{F}$ of $\mathbb{T}^2\subset S^3$.
\end{enumerate} 

From (G\ref{metrik2}), (G\ref{metrik3}) and equation \eqref{eq:1} we conclude
\begin{enumerate}[(G3')]
\item \label{metrik3a}For any sphere ${S^3(\rho)}= d^{-1}(\rho)$, $\rho \in ]0,2] \setminus\{1\}$, the second fundamental form $h^{S^3(\rho)}$ with respect to $\operatorname{grad} d$ is negative definite, and on $ S^3= S^3(1)$ the zero directions of $h^{S^3}$ are precisely the vectors tangent to the irrational geodesic foliation $\mathcal{F}$.
\end{enumerate} 

Now we will prove
\begin{proposition}\label{prop}
  Suppose $g$ is a Riemannian metric on $B^4$ satisfying conditions (G\ref{metrik1})-(G\ref{metrik4}). Then there exists a complete (non-constant) $g$-geodesic $c:\mathbb{R} \to B^4$, but no closed $g$-geodesic in $B^4$. Moreover, $\partial B^4= S^3(2)$ is strictly convex.
\end{proposition} 

\begin{proof}
  Note first that by conditions (G\ref{metrik3}') and (G\ref{metrik4}) the geodesics of the irrational foliation are complete $g$-geodesics contained in $\mathbb{T}^2 \subset B^4$, cf.~the discussion at the beginning of Section \ref{sec:convex}. Next we will show that there are no closed $g$-geodesics in $B^4$. So, let us assume that there exists a closed $g$-geodesic $c: S^1 \to B^4$. Using properties (G\ref{metrik2}), (G\ref{metrik3}) and Fact~\ref{fact:konvex} we conclude that $c$ lies in the euclidean sphere $S^3$ and that $c$ is a leaf of the irrational foliation $\mathcal{F}$ of $\mathbb{T}^2\subset S^3$. This contradicts our assumption that $c$ is closed.
\end{proof}

Now we describe how one can construct a Riemannian metric $g$ on $B^4$ that satisfies properties (G\ref{metrik1})-(G\ref{metrik4}). We consider the coordinate system
\begin{align}
  \label{eq:karte}\nonumber
  F:&\,]0,2[\times \left]0,\pi/2\right[\times \mathbb{R}^2 \longrightarrow \quad B^4 \\
&\qquad (\rho,\,\psi,\,\varphi,\,\theta)\qquad\longmapsto \left(\begin{array}{ccc}
    \rho & \cos \psi & \sin \varphi\\
    \rho & \cos \psi & \cos \varphi\\
    \rho & \sin \psi & \sin \theta\\
    \rho & \sin \psi & \cos \theta\\
  \end{array}\right).
\end{align}
For $\rho=1$ and $\psi=\nicefrac{\pi}{4}$ the coordinates $\varphi$ and $\theta$ describe the Clifford torus, i.e. $\mathbb{T}^2= F(\{1\}\times \{\nicefrac{\pi}{4}\}\times \mathbb{R}^2)$. We denote the induced coordinate vectors on $\operatorname{im} (F):= F(]0,2[\times \left]0,\nicefrac{\pi}{2}\right[\times \mathbb{R}^2 )$ by $\partial_\rho$, $\partial_\psi$, $\partial_\varphi$, $\partial_\theta$. They form a $g_0$-orthogonal frame on $\operatorname{im} (F)$, and the metric $g_0$ is given in these coordinates by the diagonal matrix
\[\operatorname{diag}(1,\rho^2,\rho^2 \cos^2\psi,\rho^2\sin^2 \psi).\]
This shows, in particular, that $F\vert_{\{1\}\times \{\nicefrac{\pi}{4}\}\times \mathbb{R}^2}$ is - up to the constant factor $\frac{1}{\sqrt{2}}$ -  an isometric covering map with group of deck transformation $2 \pi \mathbb{Z} \times 2 \pi \mathbb{Z}$. For fixed $\alpha \in \mathbb{R}\setminus \mathbb{Q}$ we consider the vectorfield $Y=\partial_\varphi + \alpha\, \partial_\theta$. The restriction of $Y$ to the torus $\mathbb{T}^2$ is tangent to an irrational geodesic foliation and the vector field $Z:= \alpha \, \tan \psi \, \partial_\varphi - \cot\psi \,\partial_\theta$ completes $\partial_\rho$, $\partial_\psi$ and $Y$ to an orthogonal frame on $\operatorname{im} (F)$.
We define a new metric $g$ on $\operatorname{im} (F)$ by requiring that the vectorfields $\partial_\rho$, $\partial_\psi$, $Y$ and $Z$ are pairwise $g$-orthogonal and by setting for $x= F(\rho,\psi,\varphi,\theta)$
\begin{align*}
  \label{eq:neuemetrik}
  g(\,\partial_{\rho},\partial_{\rho}\,)_{\vert x} &=g_0(\partial_{\rho},\partial_{\rho})_{\vert x} = 1, \\
  g(\partial_\psi,\partial_\psi)_{\vert x} &= g_0(\partial_\psi,\partial_\psi)_{\vert x}= \rho^2,\\
  g(\,Y,Y\,)_{\vert x} &=R(\rho,\psi),\\
  g(\,Z,Z\,)_{\vert x} &=g_0(Z,Z)_{\vert x} =  \rho^2 (\cos^2 \psi + \alpha^2 \sin^2 \psi),
\end{align*}
where the function $R \in C^\infty( ]0,2[\times]0,\nicefrac{\pi}{2}[,\mathbb{R}^+)$ is chosen such that the following conditions are fulfilled: 

\begin{enumerate}[(R1)]
\item \label{R1} $R(\rho,\psi) = \rho^2(cos^2\psi + \alpha^2 \sin^2 \psi)$, \newline
 if $(\rho,\psi) = (1,\nicefrac{\pi}{4})$ or $(\rho, \psi)\in \left(]0,2[\, \times\, ]0,\nicefrac{\pi}{2}[\right)\setminus \left([\nicefrac{1}{2},\nicefrac{3}{2}]\times [\nicefrac{\pi}{8},\nicefrac{3\pi}{8}]\right)$
\item \label{R2} $\frac{\partial}{\partial\rho} R(\rho,\psi)> 0$ if $(\rho,\psi) \not= (1,\nicefrac{\pi}{4})$ 
\item \label{R3} $DR_{|(1,\nicefrac{\pi}{4})} = 0 $ 
\end{enumerate}

For completeness, we will construct such a function $R$ in the appendix. First note that condition (R\ref{R1}) ensures that $g$ coincides with the standard metric $g_0$ outside the tubular neighborhood of $\mathbb{T}^2$ given by the image of $[\nicefrac{1}{2},\nicefrac{3}{2}]\times [\nicefrac{\pi}{8},\nicefrac{3\pi}{8}] \times \mathbb{R}^2$ under $F$. Therefore, the standard metric extends $g$ to a smooth metric on all of $B^4$. 

\begin{proposition}
 The metric $g$ fulfills conditions (G\ref{metrik1})-(G\ref{metrik4}).  
\end{proposition}
\begin{proof}
First note that condition (G\ref{metrik1}) follows from condition (R\ref{R1}). Next, our definition of $g$ directly implies that $\partial_\rho$ is the $g$-gradient of the euclidean distance $d$ from zero. Hence $d$ is a distance function also with respect to $g$. Thus, by the discussion in Section \ref{sec:convex}, we can calculate its $g$-Hessian on $\operatorname{im}(F)$ with equations \eqref{eq:2} and \eqref{eq:1}. As $\operatorname{grad} \, d\vert_{\operatorname{im}(F)}= \partial_\rho$ commutes with $\partial_\psi$, $Y$ and $Z$, we obtain for $V,W \in \{ \partial_{\psi}, Y, Z\}$
\begin{equation*}
  \label{eq:Hesse2}
  \frac{d}{dt}_{|t=t_0}(\Phi_t^* g)_{|p} (V,W)=  \frac{d}{dt}_{|t=t_0}g_{\Phi_t(p)}(V_{|\Phi_t(p)}, W_{|\Phi_t(p)}),
\end{equation*}
where $\Phi_t$ denotes the gradient flow of $d$. Remember that on $\on{im}(F)$ the flow lines of $\Phi$ are the $\rho$-coordinate lines. Using the preceding equation and equations \eqref{eq:2} and \eqref{eq:1} we see that on $\operatorname{im} (F)$ the matrix of the $g$-Hessian of $d$ with respect to the frame $\partial_\rho$, $\partial_\psi$, $Y$, $Z$ is the diagonal matrix given by 
\begin{equation}
  \label{eq:hessian}
  \operatorname{diag}\left(0,\rho, \frac{1}{2} \frac{\partial}{\partial \rho} R(\rho,\psi), \rho (\cos^2 \psi + \alpha^2 \sin^2 \psi)\right).
\end{equation}
Now, condition (G\ref{metrik3}) follows immediately from (R\ref{R2}) and (R\ref{R3}). Since the metric coincides with the standard metric in a neighborhood of $B^4 \setminus \operatorname{im}(F)$ and the Hessian of $d$ is positive semidefinite on $\operatorname{im}(F)$, the function $d$ is convex everywhere. So, also condition (G\ref{metrik2}) is proven. Finally, to prove (G\ref{metrik4}), we consider the projection $\pi_\psi:S^3\cap \operatorname{im} (F)\to ]0, \nicefrac{\pi}{2}[$, \, $F(1,\psi,\varphi,\theta)\mapsto \psi$. This provides a distance function with gradient $\partial_\psi$ whose gradient flowlines are given by the coordinate lines of $\psi$. Now we calculate the second fundamental form $h^{\mathbb{T}^2}$ of $\mathbb{T}^2$ in $S^3$ with respect to $\partial_\psi$, using equation \eqref{eq:2}. Then $[Y,\partial_\psi]=0$ and condition (R\ref{R3}) imply:
\begin{equation*}
  h^{\mathbb{T}^2} (Y,Y) = -\frac{1}{2}\frac{\partial}{\partial \psi} R\, (1,\nicefrac{\pi}{4}) = 0.
\end{equation*}
This completes the proof.
\end{proof}

\begin{remark}
  The construction above can easily be generalized to balls $B$ of dimension $n\geq 5$. The construction yields a Riemannian metric on $B$ fulfilling properties (G\ref{metrik1})-(G\ref{metrik4}) with the obvious modifications of the dimension.
\end{remark}

\section{An answer to a question by Brian White}
\label{sec:appl}

As mentioned in the introduction, our example is related to isoperimetric inequalities in Riemannian manifolds. Brian White \cite{White} showed that an isoperimetric inequality holds for minimal hypersurfaces (or -more generally- for codimension one varifolds) in a compact, connected Riemannian manifold $\tilde{M}$ with mean-convex boundary if $\operatorname{dim}(\tilde{M})< 7$ and if there does not exist a smooth, closed, embedded minimal hypersurface $N \subset\tilde{M}$ (The same conclusion is true if $\operatorname{dim} (\tilde{M})\geq 7$, provided one replaces ``smooth'' by ``smooth except for a singular set of Hausdorff dimension at most $\operatorname{dim} (\tilde{M}) -7$''). 
 
An isoperimetric inequality in higher codimension is obtained in \cite[Theorem 2.3]{White} under the stronger condition, that there does not exist any nonzero, stationary varifold of the same codimension. In this context, Brian White asks (cf. \cite[Remark 2.8]{White}), whether the existence of a nonzero, stationary $k$-varifold in a compact, $k$-convex Riemannian manifold $N$ implies the existence of a nonzero, stationary, integral $k$-varifold in $N$. For a brief introduction to varifolds on Riemannian manifolds see \cite[Appendix]{White}.

In the following Proposition we answer this question in the negative for codimesion larger than $2$. Starting with an arbitrary closed, connected, $m$-dimensional Riemannian manifold $(M,g')$ we consider the product metric $\gt= g'\oplus g$ on $\tilde{M}=M\times B$, where $B$ is a closed ball of dimension $n\geq 4$ and $g$ a Riemannian metric on $B$ fullfilling (G\ref{metrik1})-(G\ref{metrik4}), cf. Section \ref{example}. Then $\partial \tilde{M}= M \times \partial B$ has the following convexity property. The second fundamental form of $\partial \tilde{M}$ with respect to the inward pointing unit normal is positive semi-definite, and its kernel consists of the vectors tangent to the factor $M$. In Proposition \ref{prop:unique} we will show that $(\tilde{M},\gt)$ contains a unique stationary, $(m+1)$-dimensional varifold $V_0$ of unit mass, and, in Fact \ref{fct:notrect}, that $V_0$ is not rectifiable and, hence, not integral. This provides a negative answer to the question posed in \cite[Remark 2.8]{White}. It is easy to see that $(\tilde{M},\gt)$ does not either admit an isoperimetric inequality for $(m+1)$-dimensional submanifolds with boundary: Denoting, as before, by $c:\R \to \mathbb{T}^2 \subset B$ a $g$-geodesic that is dense on the Clifford torus $\mathbb{T}^2$, we consider the totally geodesic submanifolds $M_n=M\times c([-n,n])$ in $\tilde{M}$. They satisfy $\on{vol}_{m+1}(M_n)=2n \on{vol}_m(M)$, while $\on{vol}_m(\partial M_n)=2\on{vol}_m(M)$. This contradicts the existence of an isoperimetric inequality for $(m+1)$-dimensional minimal submanifolds (with boundary) in $\tilde{M}$.  
\begin{remark}
  According to B. White's proof of \cite[Theorem 2.3]{White} any limit of the varifolds induced by the $M_n$, normalized so as to have mass one, is a non-zero, stationary, $(m+1)$-dimensional varifold. It is easy to see (and follows from Proposition \ref{prop:unique}) that in our case there is a unique limit varifold and that this is equal to $V_0$.
\end{remark}

Next we describe the $(m+1)$-varifold $V_0$ in $\tilde{M}$: A general $(m+1)$-varifold in $\tilde{M}$ is a finite Borel measure on the total space of the Grassmann bundle $\pi:G_{m+1}(\tilde{M}) \to \tilde{M}$. The support of $V_0$ is the subset $\F$ of $G_{m+1}(\tilde{M})$ given by 
 \begin{equation*}
    \F= \{T_pM\times T_q\mathcal{F} \mid (p,q) \in  M\times \mathbb{T}^2 \},
  \end{equation*}
where $\mathcal{F}$ is the foliation of the Clifford torus $\mathbb{T}^2 \subset B$ defined in (G\ref{metrik3}). In particular, $\pi\vert_{\F}$ is one-to-one. Now $V_0$ is the pushforward of the normalized Riemannian volume of $M \times \mathbb{T}^2$, i.e. $V_0= (\pi\vert_{\F}^{-1})_\# \on{vol}_{M \times \mathbb{T}^2}$. 

In particular, the weight measure $\mu_{V_0}$ of $V_0$ is the normalized Riemannian volume $\on{vol}_{M\times \mathbb{T}^2}$ of the $(m+2)$-dimensional submanifold $M \times \mathbb{T}^2$. This implies that the $(m+1)$-density of $\mu_{V_0}$ is identically zero. 

For rectifiable $(m+1)$-varifolds $V$ the weight measure $\mu_V$ has an approximate tangent space for $\mu_V$ almost every point and hence its $(m+1)$-density is positive $\mu_V$-almost everywhere, cf. \cite[\S 15]{Simon}. Since the $(m+1)$-density of $\mu_{V_0}$ vanishes, we conclude 

\begin{fact}\label{fct:notrect}
  The $(m+1)$-varifold $V_0$ is not rectifiable. 
\end{fact}

 Here is the main result of this section. 

 \begin{proposition}\label{prop:unique}
   Let $(\tilde{M},\gt)$ and $V_0$ be as above. Then $V_0$ is stationary, and $V_0$ is the only stationary $(m+1)$-varifold of mass one in $(\tilde{M},\gt)$. 
 \end{proposition}

 \begin{remark}
   Statement and proof of Proposition \ref{prop:unique} include the case $\on{dim}(M)=m=0$. In this case the only stationary, unit mass $1$-varifold in $B$ is the stationary, non-rectifiable $1$-varifold $V_0$ with support on the tangent vectors to the irrational geodesic foliation $\mathcal{F}$ of $\mathbb{T}^2$ (see the description of $V_0$ above).  
 \end{remark}

We first recall the following well known fact from ergodic theory:

\begin{fact}(cf. \cite[p. 69]{Cornfeld})\label{fkt:flow}
  Suppose ${T_t}$ is the one-parameter group of translations on the standard torus $\mathbb{R}^2/(2 \pi \mathbb{Z})^2$ given by $[(x_1,x_2)] \mapsto [(x_1+\alpha_1 t,x_2+\alpha_2 t)]$ with $\alpha_1$ and $\alpha_2$ rationally independent. Then the flow $T_t$ is uniquely ergodic, i.e. the Lebesgue measure $\mu$ on $\mathbb{R}^2/(2 \pi \mathbb{Z})^2$ is the -- up to scale -- unique $T_t$-invariant Borel measure on $\mathbb{R}^2/(2 \pi\mathbb{Z})^2$. 
\end{fact}

\begin{corollary}\label{cor:unique}
   Let $\bar{Y}$ be the unit vector field on $\mathbb{T}^2$ tangent to $\mathcal{F}$, that is given by the normalisation of $Y\vert_{\mathbb{T}^2}$, cf. Section \ref{example}, and denote by $\varphi^{\bar{Y}}_t$ its flow. Then the Riemannian area  $\on{vol}_{\mathbb{T}^2}$ is the -- up to scale -- unique Borel measure on $\mathbb{T}^2$ that is invariant under $\varphi^{\bar{Y}}_t$. 
\end{corollary}

\begin{proof}
The norm of the vectorfield $Y=\partial_\varphi + \alpha\, \partial_\theta$ is constant on $\mathbb{T}^2$, and we denote it by $a= |Y\vert_{\mathbb{T}^2}|= \frac{1}{\sqrt{2}}\sqrt{1+\alpha^2}$. Now, the covering map $\rho: \R^2\to\mathbb{T}^2 $, $\rho(x_1,x_2) = F(1,\frac{\pi}{4},x_1,x_2)$ induces a diffeomorphism $\tilde{\rho}: \R^2/(2 \pi \mathbb{Z})^2 \to\mathbb{T}^2$ conjugating the irrational linear flow $T_t$ from Fact \ref{fkt:flow} with $\alpha_1=\frac{1}{a}$ and $\alpha_2=\frac{\alpha}{a}$ to the flow $\varphi^{\bar{Y}}_t$. So, by Fact \ref{fkt:flow}, the push-forward $\tilde{\rho}_\#\mu$ of the Lebesgue measure $\mu$ on $\R^2/(2 \pi \mathbb{Z})^2$ is the -- up to scale -- unique $\varphi^{\bar{Y}}_t$--invariant Borel measure on $\mathbb{T}^2$. On the other hand, $\tilde{\rho}_\#\mu$ equals $\on{vol}_{\mathbb{T}^2}$ up to a factor since $\tilde{\rho}$ is a homothety. 
\end{proof}

We first give a short outline of the proof of Proposition \ref{prop:unique}. In Step 1 we calculate that $V_0$ is indeed stationary, see also Remark \ref{rmk:invarg}. In Step 2 and 3 we consider an arbitrary nonzero, stationary $(m+1)$-varifold $V$ in $\tilde{M}$. In Step 2 we show that its support is contained in the set $\F\subset G_{m+1}(\tilde{M})$. This relies on the convexity properties of the spheres $S^3(\rho) \subset B$, cf. Section \ref{metrik3}. In the last step, we use the Constancy Theorem \cite[41.2(3)]{Simon} to prove that the weight measure $\mu_V$ of $V$ has a product structure. Then the unique ergodicity of the flow $\varphi_t^{\bar{Y}}$ can be used to show that $\mu_V$ is indeed proportional to the product measure $\on{vol}_M\otimes \on{vol}_{\mathbb{T}^2}$. This proves that $V=\lambda V_0$ for some $\lambda> 0$.
\newpage
\begin{proof}[Proof of Proposition \ref{prop:unique}:]
\hfill

{\bfseries Step 1:} Here we prove that $V_0$ is stationary.

\noindent{}We recall that the vectorfield $\bar{Y}$ (see Corollary \ref{cor:unique} for the definition of $\bar{Y}$) is parallel, and spans $T_q\mathcal{F}$ at every point $q$ of $\mathbb{T}^2$. We decompose any vectorfield $X$ on $\tilde{M}$ as a sum $X(p,q)=X^1_q(p)+ X^2_p(q)$, where $X^1_q(p) \in T_pM$ and $X^2_p(q) \in T_qB$. So, by the special character of the Levi Civita connection of a Riemannian product, we obtain for every $(p,q)\in M \times \mathbb{T}^2$:
\begin{equation}
  \label{eq:div}
  \begin{array}{rl}
    \on{div}_{T_pM\times T_q\mathcal{F}}X =& \left.\on{div}_M (X^1_q)\right|_p + \left.g(\nabla_{\bar{Y}}X_p^2,\bar{Y})\right|_{q}\\
  =& \left.\on{div}_M (X^1_q)\right|_p + \left. \frac{d}{dt}\right|_{t=0}\,g(X_p^2,\bar{Y})\circ \varphi^{\bar{Y}}_t(q),
  \end{array}
\end{equation}
where $\varphi^{\bar{Y}}_t$ denotes the flow of $\bar{Y}$. Now the Gauss Theorem and the invariance of the volume of the flat torus under $\varphi^{\bar{Y}}_t$, cf. Corollary \ref{cor:unique}, imply that
\begin{align*}
  \delta V_0 (X) &= \int\limits_{M\times \mathbb{T}^2}\on{div}_{T_pM\times T_q\mathcal{F}}X \,\, d\mu_{V_0}(p,q)\\
  &= \int\limits_{\mathbb{T}^2} \int\limits_{M} \left.\on{div}_M (X^1_q)\right|_p \,\on{dvol}_{M}(p) \on{dvol}_{\mathbb{T}^2}(q)\\
  &\qquad\quad + \int\limits_{M}\int\limits_{\mathbb{T}^2}   \left.  \frac{d}{dt}\right|_{t=0}\,g(X_p^2,\bar{Y})\circ \varphi^{\bar{Y}}_t(q) \, \on{dvol}_{\mathbb{T}^2}(q)\on{dvol}_{M}(p)\\   
 &= 0+ \int\limits_{M} \left. \frac{d}{dt}\right|_{t=0}\,\left(\int_{\mathbb{T}^2}   \left. g(X_p^2,\bar{Y})\right|_{q}  d(\varphi^{\bar{Y}}_{t})_\#\on{vol}_{\mathbb{T}^2}(q)\right) \on{dvol}_{M}(p) \\
& =0.
\end{align*}
So $V_0$ is stationary.\vspace{1ex}

Now, we consider an arbitrary nonzero, stationary $(m+1)$--varifold $V$ in $\tilde{M}$. 

{\bfseries Step 2:} First, we prove that the varifold $V$ has support in $\F$.

\noindent{}We consider $f:\tilde{M}\to \mathbb{R}_{\geq 0}$, $(p,q)\mapsto d^2(q)$, where $d(q)$ denotes the (euclidean) distance from $q\in B$ to $0 \in B$, cf. Section \ref{example}. Note that (G\ref{metrik2}) and (G\ref{metrik3}) imply the following:
If $(v,w) \in T_pM\times T_qB$ then $\nabla^2f((v,w),(v,w))>0$ except in the following two cases
\begin{itemize}
\item $w=0$, or
\item $q \in \mathbb{T}^2$ and $w \in T_q\mathcal{F}$.
\end{itemize}

Now suppose $V$ is a stationary $(m+1)$-varifold in $\tilde{M}$. We test $V$ against the vectorfield $X=\on{grad}f$. Then we have 
\begin{equation*}
    0 = \delta V(X)=\hspace{-2ex} \int\limits_{G_{m+1}(\tilde{M})} \hspace{-2ex}\on{div}_SX \,\,dV(S)  = \hspace{-2ex}\int\limits_{G_{m+1}(\tilde{M})}  \hspace{-2ex}\on{trace}_S(\nabla^2f) \,\,dV(S). 
  \end{equation*}
The preceding discussion shows that $\on{trace}_S(\nabla^2f)>0$ except if $S \in \F$. Hence $\on{spt}(V)\subset \F$.\vspace{1ex}

{\bfseries Step 3:} We show that $\mu_V$ equals $\on{vol}_{M\times \mathbb{T}^2}$ up to a constant. 

\noindent{}First, we prove that for any Borel set $A\subset B$ there exists $c_A>0$ such that the Borel measure $\mu^A$ on $M$ defined by $\mu^A(\,\cdot\,):= \mu_V(\,\cdot\,\times A)$ is given by $c_A \cdot \on{vol}_M$.

Note that $\mu^A$ can be considered as an $m$-varifold on the $m$-dimensional manifold $M$. We will show that $\mu^A$ is a stationary $m$-varifold, and then the Constancy Theorem \cite[41.2(3)]{Simon} implies that $\mu^A$ is a multiple of the Riemannian volume measure $\on{vol}_M$ as claimed. Denote the measure $(\pi_2)_\#\mu_V$ on $B$ by $\mu_{V,2}$, where $\pi_2:M \times B \to B$ denotes the usual projection to the second component. We choose a sequence $f_n \in C^\infty_c(B)$ converging to the indicator function $\chi_A$ in $L^1(\mu_{V,2})$. This implies that $f_n \circ \pi_2$ converges to $\chi_{M\times A}$ in $L^1(\mu_V)$.  Denoting the projection $M\times B \to M$ by $\pi_1$ we calculate for every vectorfield $X$ on $M$
\begin{align*}
  \int\limits_M \on{div}_M X \,\,d \mu^A=& \int\limits_{M\times A}\on{div}_MX\circ \pi_1 \,\, d\mu_V \\
  = &\lim_{n \to \infty} \int\limits_{M \times B}(f_n\circ\pi_2) \cdot (\on{div}_MX \circ \pi_1) \,\,d\mu_V \\
  = &\lim_{n \to \infty} \int\limits_{M \times B}\on{div}_{T_pM\times T_q\mathcal{F}} ((f_n \circ \pi_2)\cdot (X \circ \pi_1)) \,\, d\mu_V(p,q), 
\end{align*}
since it follows from equation \eqref{eq:div} that 
\[\on{div}_{T_pM\times T_q\mathcal{F}} ((f_n \circ \pi_2)\cdot (X \circ \pi_1))= f_n(q) \cdot \on{div}_MX\vert_p.\] 
 Since $V$ is stationary, we know from Step 2 that $\on{spt}(V)\subset \F$. Hence
\[\delta V((f_n\circ\pi_2)\cdot(X\circ\pi_1))=\hspace{-1ex}\int\limits_{M\times \mathbb{T}^2} \on{div}_{T_pM\times T_q\mathcal{F}}((f_n\circ\pi_2)\cdot(X\circ\pi_1))\,\, d\mu_V(p,q)=0\]
for all $n \in \mathbb{N}$. Thus $\int_M \on{div}_M X \,\, d \mu^A =0$ for every vectorfield $X$ on M, i.e. the $m$-varifold defined by $\mu^A$ is stationary, and hence a multiple of $\on{vol}_M$, see \cite[41.2(3)]{Simon}. Using the abbreviation $\mu_{V,2}=(\pi_2)_\#\mu_V$ introduced above, the constant $c_A$ can be calculated as follows
\begin{equation*}
  c_A=\frac{1}{\on{vol}_M(M)}\mu^A(M)=\frac{1}{\on{vol}_M(M)} \mu_{V,2}(A).
\end{equation*}
Hence, $\mu_{V}$ is given as a product of $\on{vol}_M$ and $\mu_{V,2}$. Next, we prove that -- up to scale -- $\mu_{V,2}$ coincides with the Riemannian area $\on{vol}_{\mathbb{T}^2}$.

The idea is to show invariance of $\mu_{V,2}$ under the flow $\varphi^{\bar{Y}}_t$ of $\bar{Y}$. Then the unique ergodicity of $\varphi^{\bar{Y}}_t$ implies that $\mu_{V,2}$ is a multiple of $\on{vol}_{\mathbb{T}^2}$, cf. Corollary \ref{cor:unique}.

 We consider $f\in C^1(B)$ and $\tilde{X}= (f\bar{Y})\circ \pi_2$. Since $\tilde{X}$ is defined in a neighborhood of $\operatorname{spt}(\mu_V)$ and $V$ is stationary we have
  \begin{align*}
    0= &\delta V(\tilde{X})= \int \on{div}_S\tilde{X}\,\, dV(S).
\end{align*}
Since $\on{spt}(V)\subset \F$, equation \eqref{eq:div} implies 
\begin{align*}
   0 =&  \int\limits_{M\times \mathbb{T}^2} \on{div}_{T_pM\times T_q\mathcal{F}}\tilde{X} \,\,d\mu_V(p,q) = \int\limits_{M\times \mathbb{T}^2} g(\nabla_{\bar{Y}}f\bar{Y},\bar{Y})\circ \pi_2\,\,d\mu_V \\
   = &\int\limits_{M\times \mathbb{T}^2} \Bigl( df(\bar{Y})  + f g(\nabla_{\bar{Y}}\bar{Y},\bar{Y})\Bigr)\circ \pi_2\,\, d\mu_V\\ 
    =&  \int\limits_{\mathbb{T}^2}  df(\bar{Y})\,\, d\mu_{V,2}.
  \end{align*}
Since every function $f \in C^1(\mathbb{T}^2)$ can be extended to a $C^1$-function on $B$ we conclude that 
\[\int\limits_{\mathbb{T}^2} df(\bar{Y})\,\, d\mu_{V,2}=0\]
for all $f \in C^1(\mathbb{T}^2)$. This implies that $\mu_{V,2}$ is $\varphi^{\bar{Y}}$-invariant. For convenience, we include the simple proof. Since $\bigl(d(f\circ\varphi^{\bar{Y}}_t)\bigr)(\bar{Y}_p)= \left.\frac{d}{dt}\right|_t \,f\circ \varphi_t^{\bar{Y}}(p)$, we have for all $t >0$
\begin{align*}
  0 =&  \int\limits_{0}^{t} \int\limits_{\mathbb{T}^2}d(f\circ \varphi^{\bar{Y}}_{\tau})(\bar{Y})\,\,  d\mu_{V,2} \,\, d\tau\\
  =&\int\limits_{\mathbb{T}^2}f\,\,d(\varphi^{\bar{Y}}_{t\#}(\mu_{V,2})) - \int\limits_{\mathbb{T}^2}f\,\,d\mu_{V,2}.
\end{align*}
This together with the Borel regularity of $\mu_{V,2}$ implies the $\varphi_t^{\bar{Y}}$-invariance of $\mu_{V,2}$. Now the unique ergodicity of $\varphi^{\bar{Y}}_t$ implies our claim, cf. Corollary \ref{cor:unique}.

This completes the proof of Step 3. Together, Step 2 und Step 3 prove the claimed uniqueness of $V_0$.
\end{proof}

\begin{remark}\label{rmk:invarg}
  Actually, the calculation in Step 1 can be replaced by the following more involved argument showing that $V_0$ is stationary. Since $\tilde{M}$ does not satisfy an isoperimetric inequality for $(m+1)$-varifolds, B.~White's Theorem 2.3 from \cite{White} implies that $\tilde{M}$ contains a nonzero, stationary $(m+1)$-varifold $V$. But now the preceding two steps show that this $V$ is a nonzero multiple of $V_0$. Hence $V_0$ is stationary.
\end{remark}

\section{Appendix}
\label{sec:appendix}

\begin{lemma} \label{lem:app1}
  There is a function $R \in C^\infty\left(\left]0,2\right[\times\left]0,\nicefrac{\pi}{2}\right[,\mathbb{R}^+\right)$ that fulfills conditions (R\ref{R1})-(R\ref{R3}).
\end{lemma}

\begin{proof}
  It is easy to find a function $k \in C^\infty(\mathbb{R}^2,\mathbb{R})$ that meets conditions (R\ref{R2}) and (R\ref{R3}), and the following weakening of condition (R\ref{R1}) 
  \begin{enumerate}[(R1')]
  \item \label{R1'}$k\left(1,\nicefrac{\pi}{4}\right)= \cos^2(\nicefrac{\pi}{4})+ \alpha^2 \sin^2 (\nicefrac{\pi}{4})$.
  \end{enumerate}
For example $k(\rho, \psi)= (\psi- \nicefrac{\pi}{4})^2 \rho + (\rho-1)^3 + \cos^2(\nicefrac{\pi}{4})+ \alpha^2 \sin^2 (\nicefrac{\pi}{4})$ has these properties, but we do not need the explicit formula. In addition, we define the function $l\in C^\infty(\mathbb{R}^2,\mathbb{R}^+)$ by $l(\rho,\psi)= \rho^2(\cos^2\psi + \alpha^2 \sin^2 \psi)$. Then
\begin{equation*}
  (k-l)\left(1,\nicefrac{\pi}{4}\right)=0 \quad \text{ and }\quad \frac{\partial}{\partial \rho}(k-l)\left(1,\nicefrac{\pi}{4}\right)= - 2\left(\cos^2\left(\nicefrac{\pi}{4}\right)+ \alpha^2 \sin^2 \left(\nicefrac{\pi}{4}\right)\right) < 0.
\end{equation*}

Therefore we can find $\nicefrac{1}{2} < \rho_1 < \rho_2<1<\rho_3<\rho_4<\nicefrac{3}{2}$ and $\nicefrac{\pi}{8}< \psi_1 < \nicefrac{\pi}{4} <\psi_2 <\nicefrac{3 \pi}{8}$ such that for any $\psi \in [\psi_1,\psi_2]$
\begin{equation}
  \label{eq:gr1}
    \begin{array}{ll}
      (k-l)(\rho,\psi) >0 & \text{ if } \rho \in [\rho_1,\rho_2],\\
      (k-l)(\rho,\psi) <0 & \text{ if } \rho \in [\rho_3,\rho_4].
    \end{array}
\end{equation}
Now choose a bump function $\beta \in C^\infty(\mathbb{R}^2,[0,1])$ with support in $[\rho_1,\rho_4]\times [\psi_1,\psi_2]$, that is constantly equal to $1$ in a neighbourhood of $\left(1,\nicefrac{\pi}{4}\right)$, and has the following property for any $\psi \in [\psi_1,\psi_2]$
\begin{equation}
  \label{eq:gr2}
   \frac{\partial}{\partial \rho} \, \beta (\rho,\psi) \left\{
       \begin{array}{ll}
         \geq 0 & \text{ for } \rho \in [\rho_1,\rho_2]\\
         = 0 & \text{ for } \rho \in [\rho_2,\rho_3]\\
         \leq 0 & \text{ for } \rho \in [\rho_3,\rho_4]\\
       \end{array}\right..
\end{equation}
Since $(k-l)(1,\nicefrac{\pi}{4})=0$ we can choose the parameters $\rho_1$, $\rho_2$, $\rho_3$, $\rho_4$, $\psi_1$ and $\psi_2$ in such a way that the function
\[R:= (1-\beta)l+ \beta k=l+ \beta(k-l)\]
is positive on the open set $]0,2[ \times ]0,\nicefrac{\pi}{2}[$. Obviously the restriction of $R$ to $]0,2[ \times ]0,\nicefrac{\pi}{2}[$ meets conditions (R\ref{R1}) and (R\ref{R3}). To finish the proof we check the monotonicity condition (R\ref{R2}):
\begin{eqnarray*}
  \frac{\partial}{\partial\rho} R &= & (1-\beta)\frac{\partial}{\partial\rho} l  +\beta\frac{\partial}{\partial\rho}k + (k-l)\frac{\partial}{\partial\rho}\beta,
\end{eqnarray*}
where the sum of the first two terms is positive if $(\rho,\psi)\not= (1,\nicefrac{\pi}{4})$ and the last term is nonnegative as \eqref{eq:gr1} and \eqref{eq:gr2} show.
\end{proof}

\thanks{
Acknowledgements: We thank Eugene Gutkin and St\'{e}phane Sabourau for useful comments. This work was partially supported by the DFG Collaborative Research Center SFB TR 71.}                   

\bibliographystyle{amsplain}

\bibliography{geodesic}

\providecommand{\bysame}{\leavevmode\hbox to3em{\hrulefill}\thinspace}
\providecommand{\MR}{\relax\ifhmode\unskip\space\fi MR }
\providecommand{\MRhref}[2]{%
  \href{http://www.ams.org/mathscinet-getitem?mr=#1}{#2}
}
\providecommand{\href}[2]{#2}
\begin{thebibliography}{1}

\bibitem{Bangert}
Victor Bangert, \emph{On the existence of closed geodesics on two-spheres},
  Internat. J. Math. \textbf{4} (1993), no.~1, 1--10. \MR{MR1209957
  (94d:58036)}

\bibitem{Birkhoff}
George~D. Birkhoff, \emph{Dynamical systems}, American Mathematical Society
  Colloquium Publications, Vol. VIII, American Mathematical Society,
  Providence, R.I., 1927.

\bibitem{Colding}
Tobias~H. Colding and William~P. Minicozzi, II, \emph{Minimal surfaces},
  Courant Lecture Notes in Mathematics, vol.~4, New York University Courant
  Institute of Mathematical Sciences, New York, 1999. \MR{MR1683966
  (2002b:49072)}

\bibitem{Cornfeld}
I.~P. Cornfeld, S.~V. Fomin, and Ya.~G. Sina{\u\i}, \emph{Ergodic theory},
  Grundlehren der Mathematischen Wissenschaften [Fundamental Principles of
  Mathematical Sciences], vol. 245, Springer-Verlag, New York, 1982, Translated
  from the Russian by A. B. Sosinski{\u\i}. \MR{832433 (87f:28019)}

\bibitem{Franks}
John Franks, \emph{Geodesics on {$S^2$} and periodic points of annulus
  homeomorphisms}, Invent. Math. \textbf{108} (1992), no.~2, 403--418.
  \MR{MR1161099 (93f:58192)}

\bibitem{Petersen}
Peter Petersen, \emph{Riemannian geometry}, Graduate Texts in Mathematics, vol.
  171, Springer-Verlag, New York, 1998. \MR{MR1480173 (98m:53001)}

\bibitem{Simon}
Leon Simon, \emph{Lectures on geometric measure theory}, Proceedings of the
  Centre for Mathematical Analysis, Australian National University, vol.~3,
  Australian National University Centre for Mathematical Analysis, Canberra,
  1983. \MR{756417 (87a:49001)}

\bibitem{White}
Brian White, \emph{Which ambient spaces admit isoperimetric inequalities for
  submanifolds?}, J. Differential Geom. \textbf{83} (2009), no.~1, 213--228.
  \MR{MR2545035}

\end{thebibliography}

\end{document}